\documentclass[11pt]{amsart}
\usepackage{amsmath}
\usepackage{amssymb}
\usepackage{graphicx}

\newtheorem{theorem}{Theorem}[section]
\newtheorem{corollary}[theorem]{Corollary}
\newtheorem{lemma}[theorem]{Lemma}

\theoremstyle{definition}

\newtheorem{remark}[theorem]{Remark}
\numberwithin{equation}{section}
\email{\tt mario.milman@gmail.com}
\input{tcilatex}

\begin{document}
\title[Poincar\'{e}-Sobolev inequalities via Garsia-Rodemich spaces ]{A note
on self improvement of Poincar\'{e}-Sobolev inequalities via Garsia-Rodemich
spaces}
\author[M. Milman]{Mario Milman}
\address{Instituto Argentino de Matematica, Argentina}

\begin{abstract}
We use the characterization of weak type inequalities via Garsia-Rodemich
conditions to show self improving properties of Poincar\'{e}-Sobolev
inequalities in a very general context.
\end{abstract}

\maketitle

\section{Introduction and Summary}

In this note we develop a new method to prove self-improving inequalities
involving oscillations through the use of Garsia-Rodemich spaces. Although
we shall apply the method to known self improving results concerning
classical Sobolev-Poincar\'{e} inequalities\footnote{%
We refer to \cite{milbmo} for the corresponding study of self-improvement of 
$BMO$ inequalities.}, we believe that our method can be useful in other
contexts as well.

One of basic problems we face here can be simply described as follows: How
can we extract information about the size of a function in terms of its
oscillations? The main ideas are classical and the fundamentals go back, at
least, to the seminal papers of Calder\'{o}n-Zygmund.

The origin of the self improving results considered in this note goes back
to the classical paper by John-Nirenberg \cite{jn}, where they introduced
the space $BMO$. Their methods were later refined by many authors (cf. \cite%
{berkovi} and the references therein). Somewhat less known are some ideas
that were developed by Garsia-Rodemich \cite{garro}. One possible reason
that the methods of \cite{garro} are less known to the community of self
improvers is the fact that the main objective of \cite{garro} lies
elsewhere, moreover, the relevant results for us are only sketched at the
end of \cite{garro}, and then only in the one dimensional case.

We addressed some of these issues in \cite{milbmo} where, in particular, we
extended the Garsia-Rodemich embedding to the $n-$dimensional case (cf.
Theorem \ref{teomarkao}).

In this note we use Garsia-Rodemich spaces to study Poincar\'{e}
inequalities in a very general context.

We shall start by recalling a construction of John-Nirenberg \cite{jn}. Let $%
Q_{0}\subset\mathbb{R}^{n},$ be a fixed cube\footnote{%
A \textquotedblleft cube" in this paper will always mean a cube with sides
parallel to the coordinate axes.}, and let 
\begin{align*}
P(Q_{0}) & =\{\{Q_{i}\}_{i\in N}:\text{countable families of subcubes }%
Q_{i}\subset Q_{0},\text{ } \\
& \text{with pairwise disjoint interiors}\}.
\end{align*}
Let $1<p<\infty.$ The John-Nirenberg spaces $JN_{p}(Q_{0})$ consist of all
functions $f\in L^{1}(Q_{0})$ such that\footnote{%
Here $f_{Q}=\frac {1}{\left\vert Q\right\vert }\int_{Q}fdx.$} (cf. \cite{jn}%
, \cite{torchinsky})%
\begin{equation}
\left\Vert f\right\Vert _{JN_{p}(Q_{0})}=\sup_{\{Q_{i}\}_{i\in N}\in
P(Q_{0})}\left\{ \dsum \limits_{i}\left\vert Q_{i}\right\vert \left( \frac{1%
}{\left\vert Q_{i}\right\vert }\int_{Q_{i}}\left\vert f-f_{Q_{i}}\right\vert
dx\right) ^{p}\right\} ^{1/p}<\infty.  \label{lajnp}
\end{equation}
John-Nirenberg \cite{jn} go on to show that%
\begin{equation}
JN_{p}(Q_{0})\subset L(p,\infty)(Q_{0}).  \label{ellimite}
\end{equation}
Thus, the $JN_{p}(Q_{0})$ condition implies the following \textquotedblleft
self-improvement": For $f\in L^{1},$ control of its $L^{1}$ oscillations, $%
\frac{1}{\left\vert Q\right\vert }\int_{Q}\left\vert f-f_{Q}\right\vert dx,$
as prescribed by (\ref{lajnp}), allows us to conclude that $f$ belongs to
the better space $L(p,\infty)$ (cf. (\ref{ellimite}))$.$ When $p\rightarrow
\infty,$ then, informally, we have $JN_{p}(Q_{0})\rightarrow BMO(Q_{0}),$
and the corresponding limiting self-improvement is expressed by the well
known John-Nirenberg Lemma \cite{jn}: Functions in $BMO$ are exponentially
integrable. Again informally, this later result corresponds to let $%
p\rightarrow\infty$ in (\ref{ellimite})$,$ and can be formulated\footnote{%
We use the somewhat unconventional notation $L(\infty,\infty)$ (also often
denoted by $W$) to define the weak-$L^{\infty}$ space (cf. \cite{bs})%
\begin{equation*}
L(\infty,\infty)=\{f:\sup_{t}\{f^{\ast\ast}(t)-f^{\ast}(t)\}<\infty\}.
\end{equation*}
Here $f^{\ast}$ denotes the non-increasing rearrangement of $f$ $\ $and $%
f^{\ast\ast}(t)=\frac{1}{t}\int_{0}^{t}f^{\ast}(s)ds,$ (cf. \cite{bs}). As
was shown in \cite{bds}$,$ $L(\infty,\infty)$ is the \textquotedblleft
rearrangement invariant hull" of $BMO.$ For further generalizations cf. \cite%
{ep}.} as%
\begin{equation}
BMO(Q_{0})\subset L(\infty,\infty)(Q_{0}).  \label{elotro}
\end{equation}

Roughly speaking, the embeddings (\ref{ellimite}), (\ref{elotro}), are the
mechanism used by John-Nirenberg to prove self improvement when we have
control of the oscillations (cf. \cite{jn}, \cite{torchinsky} and the more
recent expansive survey given in \cite{berkovi}, which contains references
to many important contributions to the topic treated in this note).

In the seventies, Garsia-Rodemich \cite{garro} introduced the closely
related spaces $GaRo_{p}(Q_{0}),1<p\leq\infty,$ whose definition we now
recall. We shall say that $f\in GaRo_{p}(Q_{0}),$ if and only if $f\in
L^{1}(Q_{0}),$ and $\exists C>0$ such that for all $\{Q_{i}\}_{i\in N}$ $\in
P(Q_{0})$ we have 
\begin{equation}
\dsum \limits_{i}\frac{1}{\left\vert Q_{i}\right\vert }\int_{Q_{i}}%
\int_{Q_{i}}\left\vert f(x)-f(y)\right\vert dxdy\leq C\left( \dsum
\limits_{i}\left\vert Q_{i}\right\vert \right) ^{1/p^{\prime}},\text{ where }%
1/p^{\prime}=1-1/p,  \label{dada}
\end{equation}
and we let%
\begin{equation*}
\left\Vert f\right\Vert _{GaRo_{p}(Q_{0})}=\inf\{C:\text{such that }(\text{%
\ref{dada}})\text{ holds}\}.
\end{equation*}
It is readily seen that (cf. \cite{milbmo}), 
\begin{equation}
JN_{p}(Q_{0})\subset GaRo_{p}(Q_{0}).  \label{la mala}
\end{equation}
A remarkable result of Garsia-Rodemich shows that (cf. \cite{garro}, and 
\cite{milbmo} for the $n-$dimensional version of the result that we use
here) as sets,%
\begin{equation}
GaRo_{p}(Q_{0})=L(p,\infty)(Q_{0}).  \label{labuena}
\end{equation}
Therefore, the gist of the matter is that the weak type spaces $L(p,\infty)$
themselves can be characterized by oscillation conditions! In other words,
the underlying method to prove (\ref{labuena}) provides an effective method
to compute the weak type norm of a function if we have control of its
oscillations, and avoids the (somewhat harder!) intermediate step of showing
the $JN_{p}(Q_{0})$ condition. As a bonus, we will also show that, when
applied to the self improvement of Poincar\'{e}-Sobolev inequalities, this
method leads to sharp results\footnote{%
Our results in this direction ought to be compared with those presented in
the recent survey \cite{berkovi} and the references therein (cf. e.g.
Corollary 3.5 in \cite{berkovi}).}.

In the last section of this note we included a brief discussion of related
methods that can be used to study self-improving inequalities; e.g. methods
based on rearrangement inequalities (cf. \cite{Kalis}), methods based on $K-$%
functional inequalities as they relate to reverse H\"{o}lder inequalities
(cf. \cite{milfen}, \cite{masmil}, \cite{mami}), and $K-$functional
inequalities applied to Poincar\'{e}-Sobolev (cf. \cite{mamicalixto} and 
\cite{corita})).

Finally, we refer to \cite{bs} for background information on rearrangements
and covering lemmas.

\section{$GaRo_{p}=L(p,\infty)$}

We consider a qualitative version of the Garsia-Rodemich \cite{garro}
equality 
\begin{equation*}
GaRo_{p}=L(p,\infty).
\end{equation*}
We start recalling the $n$ dimensional version as given in \cite{milbmo}.

\begin{theorem}
\label{teomarkao}Let $1<p<\infty,$ and let $Q_{0}\subset R^{n}$ be a fixed
cube. Then

(i) $JN_{p}(Q_{0})\subset GaRo_{p}(Q_{0}).$ In fact,%
\begin{equation*}
\left\Vert f\right\Vert _{GaRo_{p}(Q_{0})}\leq2JN_{p}(f,Q_{0}).
\end{equation*}

(ii) $GaRo_{p}(Q_{0})=L(p,\infty)(Q_{0}).$ In fact, if we let%
\begin{equation*}
\left\Vert f\right\Vert _{L(p,\infty)}^{\ast}=\sup_{t}f^{\ast}(t)t^{1/p},
\end{equation*}
then we have,%
\begin{align}
\left\Vert f\right\Vert _{GaRo_{p}(Q_{0})} & \leq\frac{2p}{p-1}\left\Vert
f\right\Vert _{L(p,\infty)}^{\ast},\text{ }  \label{ga<} \\
\sup_{t}t^{1/p}\left( f^{\ast\ast}(t)-f^{\ast}(t)\right) & \leq
2^{n/p^{\prime}+1}\left\Vert f\right\Vert _{GaRo_{p}(Q_{0})}+\left( \frac {4%
}{\left\vert Q_{0}\right\vert }\right) ^{1/p^{\prime}}\left\Vert
f\right\Vert _{L^{1}}.  \label{ga>}
\end{align}
\end{theorem}

The following form of Theorem \ref{teomarkao} will be useful for the
applications we develop in this note.

\begin{corollary}
\label{coromarkao}Let $1<p\,<\infty .$ Then,%
\begin{equation*}
\left\Vert f-f_{Q_{0}}\right\Vert _{GaRo_{p}(Q_{0})}\leq \frac{2p}{p-1}%
\left\Vert f-f_{Q_{0}}\right\Vert _{L(p,\infty )}^{\ast },
\end{equation*}%
\begin{equation}
\left\Vert f-f_{Q_{0}}\right\Vert _{L(p,\infty )}^{\ast }\leq
c(n,p)\left\Vert f-f_{Q_{0}}\right\Vert _{GaRo_{p}(Q_{0})}.  \label{bela}
\end{equation}
\end{corollary}

\begin{proof}
The first inequality follows applying (\ref{ga<}) to $f-f_{Q_{0}}.$ To prove
(\ref{bela}), let $g=f-f_{Q_{0}}.$ Then, since $g\in L^{1}(Q_{0}),$ we see
that $g^{\ast\ast}(t)\rightarrow0$ as $t\rightarrow\infty.$ Therefore, by
the fundamental theorem of calculus, we can write\footnote{%
Recall that $\frac {d}{dt}\left( g^{\ast\ast}(t)\right) =\frac{\left(
g^{\ast}(t)-g^{\ast\ast }(t)\right) }{t}.$}%
\begin{equation*}
g^{\ast\ast}(t)=\int_{t}^{\infty}\left( g^{\ast\ast}(s)-g^{\ast}(s)\right) 
\frac{ds}{s}.
\end{equation*}
Combining with $($\ref{ga>}$)$ we find%
\begin{align*}
g^{\ast\ast}(t) & \leq c\left\Vert g\right\Vert
_{GaRo_{p}(Q_{0})}\int_{t}^{\infty}s^{-1/p}\frac{ds}{s}+\left( \frac{4}{%
\left\vert Q_{0}\right\vert }\right) ^{1/p^{\prime}}\left\Vert g\right\Vert
_{L^{1}}\int_{t}^{\infty }s^{-1/p}\frac{ds}{s} \\
& =p\left( c\left\Vert g\right\Vert _{GaRo_{p}(Q_{0})}+\left( \frac {4}{%
\left\vert Q_{0}\right\vert }\right) ^{1/p^{\prime}}\left\Vert g\right\Vert
_{L^{1}}\right) t^{-1/p}.
\end{align*}

Thus,%
\begin{equation}
\left\Vert g\right\Vert _{L(p,\infty )}^{\ast }\leq \sup_{t}g^{\ast \ast
}(t)t^{1/p}\leq p\left( c\left\Vert g\right\Vert _{GaRo_{p}(Q_{0})}+\left( 
\frac{4}{\left\vert Q_{0}\right\vert }\right) ^{1/p^{\prime }}\left\Vert
g\right\Vert _{L^{1}}\right) .  \label{laboba}
\end{equation}%
Now, since $\int_{Q_{0}}g=0,$ and $\{Q_{0}\}\in P(Q_{0}),$%
\begin{align*}
\int_{Q_{0}}\left\vert g(x)\right\vert dx& =\int_{Q_{0}}\left\vert g(x)-%
\frac{1}{\left\vert Q_{0}\right\vert }\int_{Q_{0}}g(y)dy\right\vert dx \\
& =\int_{Q_{0}}\frac{1}{\left\vert Q_{0}\right\vert }\left\vert
\int_{Q_{0}}(g(x)-g(y))dy\right\vert dx \\
& \leq \frac{1}{\left\vert Q_{0}\right\vert }\int_{Q_{0}}\int_{Q_{0}}\left%
\vert g(x)-g(y)\right\vert dydx \\
& \leq \left\vert Q_{0}\right\vert ^{1/p^{\prime }}\left\Vert g\right\Vert
_{GaRo_{p}(Q_{0})}.
\end{align*}%
Inserting this information in (\ref{laboba}) we find,%
\begin{equation*}
\left\Vert g\right\Vert _{L(p,\infty )}^{\ast }\leq p\left( c+\left( \frac{4%
}{\left\vert Q_{0}\right\vert }\right) ^{1/p^{\prime }}\left\vert
Q_{0}\right\vert ^{1/p^{\prime }}\right) \left\Vert g\right\Vert
_{GaRo_{p}(Q_{0})},
\end{equation*}%
concluding the proof.
\end{proof}

For further use below let us also note the corresponding end point result
for $p=\infty$.

\begin{lemma}
\begin{equation}
GaRo_{\infty}(Q_{0})=BMO(Q_{0}).  \label{barata}
\end{equation}
\end{lemma}

\begin{proof}
First let us note that, as is well known, and readily verified (cf. \cite%
{garro}), 
\begin{equation}
f\in BMO(Q_{0})\Leftrightarrow\left\Vert f\right\Vert _{\ast}=\sup_{Q\subset
Q_{0}}\frac{1}{\left\vert Q\right\vert ^{2}}\int_{Q}\int_{Q}\left\vert
f(x)-f(y)\right\vert dxdy<\infty,  \label{laboba1}
\end{equation}
where in the expression defining $\left\Vert f\right\Vert _{\ast}$ above$,$
the $\sup$ is taken over all subcubes $Q\subset Q_{0}.$ This given, let us
suppose first that $f\in GaRo_{\infty}(Q_{0}).$ Then, since for any subcube $%
Q\subset Q_{0}$ we have $\{Q\}\in P(Q_{0}),$ it follows that%
\begin{equation*}
\frac{1}{\left\vert Q\right\vert }\int_{Q}\int_{Q}\left\vert
f(x)-f(y)\right\vert dxdy\leq\left\vert Q\right\vert \left\Vert f\right\Vert
_{GaRo_{\infty}(Q_{0})}.
\end{equation*}
Thus, by (\ref{laboba1}),%
\begin{equation*}
\left\Vert f\right\Vert _{\ast}\leq\left\Vert f\right\Vert _{GaRo_{\infty
}(Q_{0})}.
\end{equation*}
Conversely, for any $\{Q_{i}\}_{i\in N}\in P(Q_{0}),$ we can estimate%
\begin{align*}
\dsum \limits_{i}\frac{1}{\left\vert Q_{i}\right\vert }\int_{Q_{i}}%
\int_{Q_{i}}\left\vert f(x)-f(y)\right\vert dxdy & =\dsum
\limits_{i}\left\vert Q_{i}\right\vert \frac{1}{\left\vert Q_{i}\right\vert
^{2}}\int_{Q_{i}}\int_{Q_{i}}\left\vert f(x)-f(y)\right\vert dxdy \\
& \leq\left( \dsum \limits_{i}\left\vert Q_{i}\right\vert \right) \left\Vert
f\right\Vert _{\ast}.
\end{align*}
Whence,%
\begin{equation*}
\left\Vert f\right\Vert _{GaRo_{\infty}(Q_{0})}\leq\left\Vert f\right\Vert
_{\ast}.
\end{equation*}
\end{proof}

Summarizing, just like the $JN_{p}$ conditions, the $GaRo_{p}$ conditions
form a scale joining the weak type Marcinkiewicz $L(p,\infty)$ spaces and $%
BMO.$ Moreover, we have%
\begin{equation*}
JN_{p}\subset GaRo_{p}=L(p,\infty),\text{ for }p\in(1,\infty),
\end{equation*}
and%
\begin{equation*}
JN_{\infty}=GaRo_{\infty}=BMO.
\end{equation*}

\section{Poincar\'{e} Inequalities}

Let $p\in(1,\infty).$ We consider $S_{p}(Q_{0}),$ the class of functions of
functions $f\in L^{1}(Q_{0}),$ such that there exists a constant $c(f)>0,$
and $g\in L^{p}(Q_{0}),$ such that for all subcubes $Q\subset Q_{0},$%
\begin{equation*}
\frac{1}{\left\vert Q\right\vert }\int_{Q}\left\vert f-f_{Q}\right\vert
dx\leq c(f)l(Q)\left\{ \frac{1}{\left\vert Q\right\vert }\int_{Q}\left\vert
g\right\vert ^{p}dx\right\} ^{1/p},
\end{equation*}
where $l(Q)=$ length of the sides of $Q.$

The function $g$ is usually called an upper gradient of $f$ (cf. \cite{hkf}, 
\cite{haj}, \cite{Kalis} and the references therein). As is well known, with
a minor variant of this definition\footnote{%
We need to replace *cubes* by *balls*. For more information on this point
see Remark \ref{remarkmarkao} below.} one can study Poincar\'{e}
inequalities in metric spaces. In particular, the classical Euclidean $(1,p)$
Poincar\'{e} inequalities, correspond to the choice $\left\vert g\right\vert
=\left\vert \nabla f\right\vert .$

\begin{theorem}
(i) Let $1<p<n.$ Suppose that $f\in S_{p}(Q_{0})$ then, $f\in L(p^{\ast
},\infty)(Q_{0}),$ where $\frac{1}{p^{\ast}}=\frac{1}{p}-\frac{1}{n}$

(ii) If $p=n,$ then $f\in S_{n}(Q_{0})$ implies that $f\in
GaRo_{\infty}(Q_{0})=BMO(Q_{0}).$
\end{theorem}

\begin{proof}
(i) Let $\{Q_{i}\}_{i\in N}\in P(Q_{0}).$ Let $1<p<n,$ and $\frac{1}{p^{\ast}%
}=\frac{1}{p}-\frac{1}{n},\frac{1}{p^{\prime}}=1-\frac{1}{p},\frac{1}{%
(p^{\ast})^{\prime}}=1-\frac{1}{p^{\ast}}.$ Note that since $p<n,$ then $%
p^{\ast}>p,$ and we have $p^{\ast}/p>1.$ Then\footnote{%
In the course of the proof we use the fact that%
\begin{equation*}
\left\Vert \{x_{n}\}\right\Vert _{l^{p^{\ast}/p}}\leq\left\Vert
\{x_{n}\}\right\Vert _{l^{1}}.
\end{equation*}%
},%
\begin{align*}
& \dsum \limits_{i}\frac{1}{\left\vert Q_{i}\right\vert }\int_{Q_{i}}%
\int_{Q_{i}}\left\vert f(x)-f(y)\right\vert dxdy \\
& \leq2\dsum \limits_{i}\int_{Q_{i}}\left\vert f(x)-f_{Q_{i}}\right\vert dx
\\
& =2\dsum \limits_{i}\left\vert Q_{i}\right\vert \frac{1}{\left\vert
Q_{i}\right\vert }\int_{Q_{i}}\left\vert f(x)-f_{Q_{i}}\right\vert dx \\
& \leq2c(f)\dsum \limits_{i}\left\vert Q_{i}\right\vert l(Q_{i})\left\{ 
\frac{1}{\left\vert Q_{i}\right\vert }\int_{Q_{i}}\left\vert g\right\vert
^{p}dx\right\} ^{1/p} \\
& =2c(f)c_{n}\dsum \limits_{i}\left\vert Q_{i}\right\vert
^{1-1/p+1/n}\left\{ \int_{Q_{i}}\left\vert g\right\vert ^{p}dx\right\} ^{1/p}
\\
& =2c(f)c_{n}\dsum \limits_{i}\left\vert Q_{i}\right\vert
^{1-1/p^{\ast}}\left\{ \int_{Q_{i}}\left\vert g\right\vert ^{p}dx\right\}
^{1/p}
\end{align*}%
\begin{align*}
& \leq2c(f)c_{n}\left\{ \dsum \limits_{i}\left\vert Q_{i}\right\vert ^{\frac{%
1}{(p^{\ast})^{\prime}}(p^{\ast})^{\prime }}\right\} ^{1/(p^{\ast})^{\prime}}%
\left[ \left\{ \dsum \limits_{i}\left\{ \int_{Q_{i}}\left\vert g\right\vert
^{p}dx\right\} ^{p^{\ast}/p}\right\} ^{p/p^{\ast}}\right] ^{1/p} \\
& \leq2c(f)c_{n}\left\{ \dsum \limits_{i}\left\vert Q_{i}\right\vert
\right\} ^{1/(p^{\ast})^{\prime}}\left\{ \dsum
\limits_{i}\int_{Q_{i}}\left\vert g\right\vert ^{p}dx\right\} ^{1/p} \\
& \leq2c(f)c_{n}\left\{ \dsum \limits_{i}\left\vert Q_{i}\right\vert
\right\} ^{1/(p^{\ast})^{\prime}}\left\{ \int_{\cup Q_{i}}\left\vert
g\right\vert ^{p}dx\right\} ^{1/p} \\
& \leq2c(f)c_{n}\left\{ \dsum \limits_{i}\left\vert Q_{i}\right\vert
\right\} ^{1/(p^{\ast})^{\prime}}\left\{ \int_{Q_{0}}\left\vert g\right\vert
^{p}dx\right\} ^{1/p}.
\end{align*}
Thus,%
\begin{equation}
\left\Vert f\right\Vert _{GaRo_{p^{\ast}}(Q_{0})}\leq2c(f)c_{n}\left\{
\int_{Q_{0}}\left\vert g\right\vert ^{p}dx\right\} ^{1/p}.  \label{lafrume}
\end{equation}
Now, since%
\begin{align*}
\frac{1}{\left\vert Q\right\vert }\int_{Q}\left\vert
(f-f_{Q_{0}})-(f-f_{Q_{0}})_{Q}\right\vert dx & =\frac{1}{\left\vert
Q\right\vert }\int_{Q}\left\vert f-f_{Q}\right\vert dx \\
& \leq c(f)l(Q)\left\{ \frac{1}{\left\vert Q\right\vert }\int_{Q}\left\vert
g\right\vert ^{p}dx\right\} ^{1/p},
\end{align*}
we see that $g$ is also an upper gradient of $f-f_{Q_{0}}.$ Consequently, (%
\ref{lafrume}) holds for $f-f_{Q_{0}}$ and we find that%
\begin{equation*}
\left\Vert f-f_{Q_{0}}\right\Vert
_{GaRo_{p^{\ast}}(Q_{0})}\leq2c(f)c_{n}\left\{ \int_{Q_{0}}\left\vert
g\right\vert ^{p}dx\right\} ^{1/p}.
\end{equation*}
Applying (\ref{bela}) we finally arrive at 
\begin{equation*}
\left\Vert f-f_{Q_{0}}\right\Vert _{L(p,\infty)}^{\ast}\leq c(n,p,\left\vert
Q_{0}\right\vert )2c(f)c_{n}\left\{ \int_{Q_{0}}\left\vert g\right\vert
^{p}dx\right\} ^{1/p}.
\end{equation*}

(ii) Suppose that $p=n.$ We proceed as in the first part of the proof
noticing that when $p=n,$ we have $1-1/p^{\ast}=1.$ Consequently,%
\begin{align*}
\dsum \limits_{i}\frac{1}{\left\vert Q_{i}\right\vert }\int_{Q_{i}}%
\int_{Q_{i}}\left\vert f(x)-f(y)\right\vert dxdy & \leq2c(f)c_{n}\dsum
\limits_{i}\left\vert Q_{i}\right\vert \left\{ \int_{Q_{i}}\left\vert
g\right\vert ^{n}dx\right\} ^{1/n} \\
& \leq2c(f)c_{n}\left\{ \int_{Q_{0}}\left\vert g\right\vert ^{n}dx\right\}
^{1/n}\dsum \limits_{i}\left\vert Q_{i}\right\vert .
\end{align*}
Therefore,%
\begin{equation*}
\left\Vert f-f_{Q_{0}}\right\Vert
_{GaRo_{\infty}(Q_{0})}\leq2c(f)c_{n}\left\{ \int_{Q_{0}}\left\vert
g\right\vert ^{n}dx\right\} ^{1/n},
\end{equation*}
and we conclude by (\ref{barata}).
\end{proof}

\begin{remark}
The previous result shows that starting with a function in $%
S_{p}(Q_{0}),1<p<n,$ we obtain the (weak type) improvement $f\in
L(p^{\ast},\infty)(Q_{0}).$ Moreover, combining this result with Maz'ya's
self-improvement principle for weak type inequalities for the gradient%
\footnote{%
To the effect that `weak type\textquotedblright\ implies \textquotedblleft
strong type".} (cf. \cite{haj}) we obtain (the strong type) improvement: if $%
f\in S_{p}(Q_{0})$ then $f\in L^{p^{\ast}}(Q_{0})$ or even $%
L(p^{\ast},p)(Q_{0})$ (cf. \cite{mamicalixto})$.$ Since we have nothing new
to add to the known methods used to show how to self improve from weak type
to strong type, we shall not consider this issue here and refer to \cite{ta}%
, \cite{haj}, \cite{mamipus}, \cite{mamicalixto}, and the references therein.
\end{remark}

\section{Final Comments and Problems}

We will show some connections with other approaches to the self improvement
of Sobolev-Poincar\'{e} inequalities.

\subsection{Poincar\'{e} inequalities, maximal inequalities and
rearrangements}

There is a close connection between Sobolev-Poincar\'{e} inequalities,
rearrangement inequalities for gradients, weak type inequalities and maximal
operators. Consequently all of the above can be expressed in terms of
Garsia-Rodemich conditions. In this section we shall briefly explore some of
these interconnections.

Let $Q_{0}$ be a fixed cube on $R^{n}.$ Suppose that $f\in S_{1}(Q_{0}).$
Therefore, there exists a constant $c(f)\geq 0,$ and $g\in L^{1}(Q_{0}),$
such that for all $Q\subset Q_{0}$ subcubes of $Q_{0},$ we have%
\begin{equation}
\frac{1}{\left\vert Q\right\vert }\int_{Q}\left\vert f(x)-f_{Q}\right\vert
dx\leq c(f)\frac{\left\vert Q\right\vert ^{1/n}}{\left\vert Q\right\vert }%
\int_{Q}\left\vert g(x)\right\vert dx.  \label{vale1}
\end{equation}%
We now reproduce the argument in \cite{Kalis}. We first note that if (\ref%
{vale1}) holds then,%
\begin{align*}
f_{1/n}^{\#}(x)& :=\sup_{Q\backepsilon x}\frac{1}{\left\vert Q\right\vert
^{1+1/n}}\int_{Q}\left\vert f(x)-f_{Q}\right\vert dx \\
& \leq c(f)\sup_{Q\backepsilon x}\frac{1}{\left\vert Q\right\vert }%
\int_{Q}g(x)dx=c(fMg(x),
\end{align*}%
where $M$ is the maximal operator of Hardy-Littlewood. Consequently, by a
modification of an argument of \cite{bs}, we find%
\begin{align*}
f_{1/n}^{\#\ast }(t)& \leq C_{n}\left( Mg\right) ^{\ast \ast }(t) \\
& =C_{n}g^{\ast \ast }(t).
\end{align*}%
As a consequence (cf. \cite{bs}, \cite{corita}) we obtain a version of a
well known rearrangement inequality for the gradient (cf. \cite{kol}, \cite%
{bmr}, \cite{Kalis} and the references therein) 
\begin{equation}
f^{\ast \ast }(t)-f^{\ast }(t)\leq c_{n}t^{1/n}g^{\ast \ast }(t),\text{ for }%
0<t<\left\vert Q_{0}\right\vert /2.  \label{la ser3}
\end{equation}%
Note that (\ref{la ser3}) yields a weak type form of the Gagliardo-Nirenberg
inequality. Indeed, we can rewrite (\ref{la ser3}) as 
\begin{equation*}
\left( f^{\ast \ast }(t)-f^{\ast }(t)\right) \leq c_{n}t^{1/n}\frac{1}{t}%
\int_{0}^{t}g^{\ast }(s)ds,\text{ }0<t<\left\vert Q_{0}\right\vert /2,
\end{equation*}%
which readily implies (cf. \cite{milbmo})%
\begin{equation*}
\sup_{t>0}\left( f^{\ast \ast }(t)-f^{\ast }(t)\right) t^{1/n^{\prime }}\leq
c_{n}\left\Vert g\right\Vert _{L^{1}(Q_{0})}+\left( \frac{\left\vert
Q_{0}\right\vert }{2}\right) ^{\frac{1}{n^{\prime }}-1}\left\Vert
f\right\Vert _{L^{1}(Q_{0})}.
\end{equation*}%
It follows that (cf. the proof of Corollary \ref{coromarkao} above)%
\begin{equation*}
\left\Vert f-f_{Q_{0}}\right\Vert _{L(n^{\prime },\infty )}^{\ast }\leq
c(n,\left\vert Q_{0}\right\vert )\left[ c(f)\left\Vert g\right\Vert
_{L^{1}(Q_{0})}+\left( \frac{\left\vert Q_{0}\right\vert }{2}\right) ^{\frac{%
1}{n^{\prime }}-1}\left\Vert f-f_{Q_{0}}\right\Vert _{L^{1}(Q_{0})}\right] 
\end{equation*}%
and since by (\ref{vale1})%
\begin{align*}
\left\Vert f-f_{Q_{0}}\right\Vert _{L^{1}(Q_{0})}& \leq c(f)\left\vert
Q_{0}\right\vert ^{1/n}\int_{Q_{0}}\left\vert g(x)\right\vert dx \\
& =c(f)\left\vert Q_{0}\right\vert ^{1/n}\left\Vert g\right\Vert
_{L^{1}(Q_{0})},
\end{align*}%
we finally arrive at%
\begin{equation*}
\left\Vert f-f_{Q_{0}}\right\Vert _{L(n^{\prime },\infty )}^{\ast }\leq
c(n)\left\Vert g\right\Vert _{L^{1}(Q_{0})}.
\end{equation*}%
Of course this weak type version of the Gagliardo-Nirenberg inequality can
be now rewritten using Garsia-Rodemich conditions via (\ref{bela}).

The analysis for $(q,p)$ Poincar\'{e} inequalities follows the same pattern
(cf. \cite{Kalis}). For example, suppose that $f\in S_{p}(Q_{0}).$ Then
there exists $c(f)>0,$ and $g$ upper gradient of $f$, such that for all
subcubes $Q\subset Q_{0},$%
\begin{equation*}
\left( \frac{1}{\left\vert Q\right\vert }\int_{Q}|f(x)-f_{Q}|^{q}dx\right)
^{1/q}\leq c(f)\left\vert Q\right\vert ^{1/n}\left( \frac{1}{\left\vert
Q\right\vert }\int_{Q}g^{p}(x)dx\right) ^{1/p}.
\end{equation*}%
Then%
\begin{equation*}
t^{\frac{1}{n}}\left( \frac{1}{t}\int \left[ f^{\ast }(s)-f^{\ast }(s)\right]
^{q}ds\right) ^{1/q}\leq Cc(f)[(g^{p})^{\ast \ast }(t)]^{1/p},\text{for }%
0<t<\left\vert Q_{0}\right\vert /2.
\end{equation*}%
Which again can be rewritten as a Sobolev-Poincar\'{e} weak type inequality.

\begin{remark}
\label{remarkmarkao}As was shown in \cite{Kalis} the analysis above holds in
the general setting of metric spaces provided with a doubling measure. In
particular, on a doubling metric measure space $(X,\mu)$ there exists $s=$
doubling order, or homogeneous dimension, such that for each ball $B\subset
X $%
\begin{equation*}
\mu(B)\geq cr(B)^{s},\text{ where }r(B)\text{ is the radius of }B.
\end{equation*}
The corresponding rearrangement inequality associated with the Poincar\'{e}
inequality%
\begin{equation*}
\frac{1}{\mu(B)}\int_{B}\left\vert f(x)-f_{B}\right\vert d\mu(x)\leq c\frac{%
\mu(B)^{1/s}}{\mu(B)}\int_{B}\left\vert g(x)\right\vert d\mu(x),
\end{equation*}
takes the form (cf. \cite{Kalis})%
\begin{equation*}
f^{\ast\ast}(t)-f^{\ast}(t)\leq c_{n}t^{1/s}g^{\ast\ast}(t),\text{ for }%
0<t<\mu(X)/2.
\end{equation*}
\end{remark}

\begin{remark}
It could be of interest to investigate the connection of the homogenous
dimension of $X$ and measures of the form $wd\mu,$ with $w$ in a Muckenhoupt
class of weights. Given the connection between $BMO$ and the $A_{p}$ classes
(cf. \cite{coif}) it would be interesting to study the connection with $BMO$
and isoperimetry. Recently, an interesting connection between isoperimetry
and $BMO$ was uncovered in \cite{bre}.
\end{remark}

\subsection{K-functional connection}

Due to its connection to maximal operators, the $K-$functional of
interpolation theory (cf. \cite{bs}) is a good tool to study self-improving
inequalities involving averages. For example, the well known equivalence of
Herz-Stein to the effect that the maximal operator of Hardy Littlewood, $M,$
can be estimated by%
\begin{equation*}
\left( Mf\right) ^{\ast}(t)\approx f^{\ast\ast}(t),
\end{equation*}
can be effectively used to prove self improving inequalities connected with
reverse H\"{o}lder inequalities (cf. \cite{milfen}, \cite{masmil}, \cite%
{mami}, and the references therein). Moreover, this is immediately connected
with the computation of the $K-$functional for the pair $(L^{1},L^{\infty}):$%
\begin{equation*}
K(t,f;L^{1},L^{\infty})=tf^{\ast\ast}(t).
\end{equation*}
In this context Gehring's self-improving inequalities can be formulated as
differential inequalities connected with the reiteration formulae of
Holmstedt (cf. \cite{milfen}).

Likewise, we believe that suitable reformulations of the $K-$functional for
the pair or the pair $(L^{1},BMO)$ (cf. \cite{bs}, \cite{corita}) can be
used to reformulate some of the results considered in this note in terms of
differential inequalities, via reiteration (cf. \cite{milfen}).

\textbf{Acknowledgement.} We are grateful to the referee for a number of
suggestions to improve the presentation of the paper. The author was
partially supported by a grant from the Simons Foundation (\TEXTsymbol{%
\backslash}\#207929 to Mario Milman)

\end{document}